\documentclass[leqno]{aadmbook}
\usepackage{amsthm}
\usepackage{amsfonts}
\usepackage{amsmath}
\usepackage{latexsym,amsfonts,mathptm}
\usepackage{epsfig,float,graphicx}
\restylefloat{figure} 
\def\Dj{\rlap{-}D}

\newtheorem{thm}{Theorem}
\newtheorem{lem}{Lemma}

\newtheorem{df}{Definition}

\newtheorem{exm}{Example}

\def\ds{\displaystyle}
\def\dzn{,\kern-0.1em,}

\input cyracc.def
 
\def\be{\begin{equation} }
\def\ee{\end{equation} }
\def\bfl{\begin{flushleft} }
\def\efl{\end{flushleft} }
\def\bfr{\begin{flushright} }
\def\efr{\end{flushright} }
\def\bc{\begin{center}}
\def\vs*{\vspace*}
\def\hs*{\hspace*}
\def\ec{\end{center}}
\def\beq{\begin{eqnarray}}
\def\eeq{\end{eqnarray}}

\def\ben{\begin{enumerate}}
\def\een{\end{enumerate}}
\def\bit{\begin{itemize}}
\def\eit{\end{itemize}}

\begin{document}

% Your \newcommands below (if there are any):

\oddsidemargin 16.5mm
\evensidemargin 16.5mm

\thispagestyle{plain}

%\begin{center}
%{\large \sc  Applicable Analysis and Discrete Mathematics}
%{\small available online at  http:/$\!$/pefmath.etf.rs }
%\end{center}

%\noindent{\small{\sc  Appl.\ Anal.\ Discrete Math.\ }{\bf x} (xxxx),
%xxx--xxx.} \hfill{\scriptsize doi:10.2298/AADMxxxxxxxx}

\vspace{5cc}
\begin{center}
{\large\bf  A short note on the order of the double reduced 2-factor  transfer digraph for rectangular grid graphs
\rule{0mm}{6mm}\renewcommand{\thefootnote}{}%Enter at least one, but not more than 3 MSCs.
% First entered MSC will be a primary one, others (at most 2) will be secondary.
\footnotetext{

\rule{2.4mm}{0mm}Keywords and Phrases:  2-factor,  transfer matrix, palindromes,  grid graphs}}

\vspace{1cc} {\large\it   Jelena \Dj oki\' c}

\vspace{1cc}
\parbox{24cc}{{\small
We prove that the order of
 the double reduced 2-factor transfer digraph ${\cal R}^{**}_{m}$ which is
needed for the  enumeration of the spanning unions of cycles   in  the
rectangular grid graph $P_m \times P_n$ ($m,n \in N$),  when $m$ is odd,
is equal to $\ds \mid  V({\cal R}^{**}_{m}) \mid =  \frac{1}{2} \left[{m+1 \choose (m-1)/2 } + {(m+1)/2  \choose \lfloor (m+1)/4  \rfloor}\right].$
 }}
\end{center}

% The paper should have at least two sections

\vspace{1.5cc}
\begin{center}
{\bf  1.  INTRODUCTION}
\end{center}
\label{sec:intro}

\vspace*{4mm}

Even though the research devoted  to the enumeration of  Hamiltonian cycles on special classes of  grid graphs of fixed width (e.g. Cartesian products of paths or cycles)
 has been going on for more than thirty years \cite{TBKS},  there still remain
 many  open  questions (for more details see \cite{BKP1,BKDP1,BKDjDP}).
 Our previous studies suggest that  the  answers  on  most of these questions lie within the structure of so-called \emph{transfer digraphs} for Hamiltonian cycles - auxiliary digraphs using which the counting is performed. The first step towards this aim
 is to investigate the structure of transfer digraphs for 2-factors which are the  natural generalization of the concept of Hamiltonian cycles.

A spanning subgraph of  a graph $G$  where each vertex has exactly two neighbors is called \emph{ 2-factor }. In other words it  represents the spanning  union of cycles. In  the case of the existence of only one cycle, this 2-factor is the Hamiltonian cycle. The systematic study on the 2-factor transfer digraphs has recently   begun \cite{DjBD1,DjDB2,DjDB3}.
The first results refer to the  following  grid  graphs.

\begin{df} \label{def:grafovi} \ \
The \textbf{rectangular (grid) graph $RG_m(n)$} and \textbf{  thick (grid) cylinder $TkC_m(n)$} ($m,n \in N$) are
 $P_m \times P_{n}$ and  $P_m \times C_n$, respectively.
\ The \textbf{Moebius strip $MS_m(n)$}  is obtained from $RG_m(n+1) = P_m \times P_{n+1}$
  by identification  of  corresponding vertices from the first and  last  column in the opposite direction without duplicating edges.
The value $m \in N$ is called the \textbf{width} of the grid graph.
\end{df}

All the three classes of these grid graphs  are called \emph{linear} ones because
the   subgraph  induced by  all vertices from any column of a such grid graph is the path $P_m$.
The 2-factor of the rectangular grid graph  $RG_7(8)$ depicted in  Figure~\ref{SSL} (a)
consists of three cycles.

\begin{figure}[H]
\begin{center}
\includegraphics[width=4.5in]{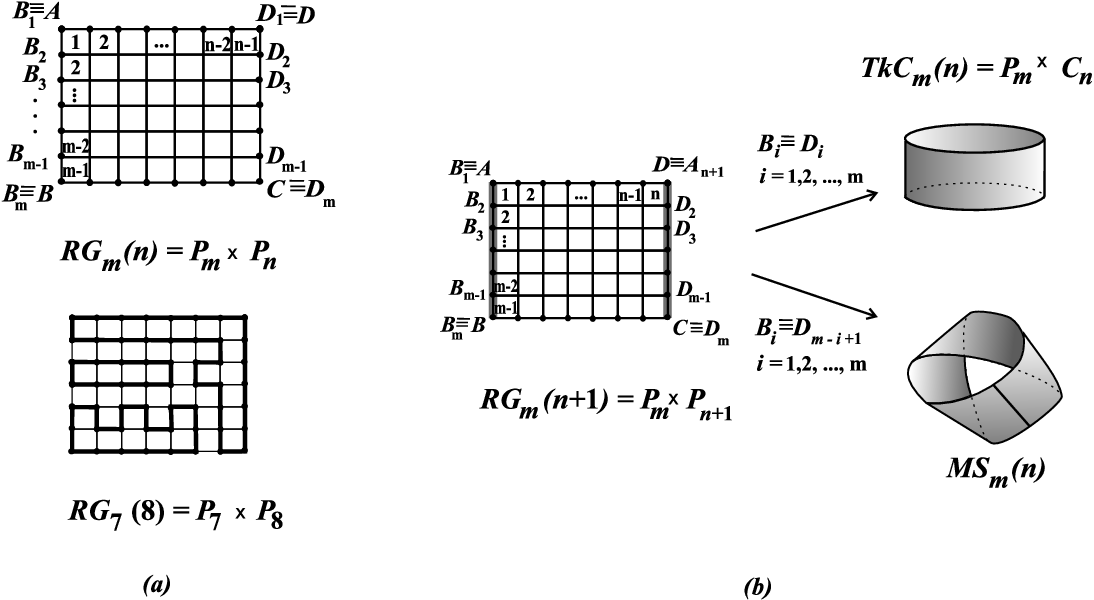}
\\ \ \vspace*{-18pt}
\end{center}
\caption{ {\bf (a)} The rectangular grid  $RG_m(n) = P_m \times P_n$ and  $RG_7(8)$  with one of its  2-factors (in bold lines);
  {\bf (b)} Constructing of the  thick cylinder  $TkC_m(n) = P_m \times C_n$ and  Moebius strip  $MS_m(n)$ from $RG_m(n + 1)$.}
\label{SSL}
\end{figure}
\unskip

Observe  a 2-factor of an arbitrary linear grid graph $G$ and its edges   which are incident to a vertex $v$ of $G$.
All the possible  arrangements  of these edges around $v$ are shown in
  Figure~\ref{CvorniKod1} (the edges in bold belong to the 2-factor).
They determine for each vertex of $G$  its \emph{code letter}.
Note that the code letters for vertices in corners of $G = RG_m(n)$  are uniquely determined and are from the set $\{ a, c, d, f \}$, while there exist
exactly three possible letters for the  other boundary vertices   of $G$.

\begin{figure}[H]
\begin{center}
\includegraphics[width=3in]{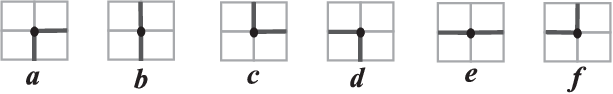}
\\ \ \vspace*{-18pt}
\end{center}
\caption{The six possible arrangements of the two edges  around any vertex with assigned   code letters.}
\label{CvorniKod1}
\end{figure}
\unskip

\begin{df} \label{def:2-factor} \cite{DjBD1}
For a given 2-factor of a linear  grid graph $G$  of width $m$  and with $m \cdot n$ vertices ($m,n \in N$),
the {\bf {\em code matrix}} $ \ds \left[ \alpha_{i,j}\right]_{m \times n}$  is
a matrix of order $m \times n$ with entries from $\{ a,b, c, d, e,f \}$ where
$\alpha_{i,j}$ is the code letter for the $i$-th vertex in $j$-th
column of $G$.
A  word over alphabet $\{ a,b,c,d,e,f \}$ of length $m$ obtained by reading  a column of the code matrix from top to down is called  {\bf {\em  alpha-word}}.
\end{df}

\begin{exm}  \label{exm:0}
For the  2-factor of $RG_7(8)$ depicted in  Figure~\ref{SSL} (a), the alpha words for its columns are $(ac)^2abc$, $e^4dce$, $e^4afe$, $e^4dce$, $e^2dfafe$,
$e^2acdbf$, $edfdb^2c$ and $db^5f$ (in that order).
\end{exm}

The possibility that two vertices are adjacent in the considered  2-factor is expressed through the two  auxiliary digraphs  ${\cal D}_{ud}$ and ${\cal D}_{lr}$
depicted  in   Figure~\ref{CvorniKod2}.

\begin{figure}[H]
\begin{center}
\includegraphics[width=3in]{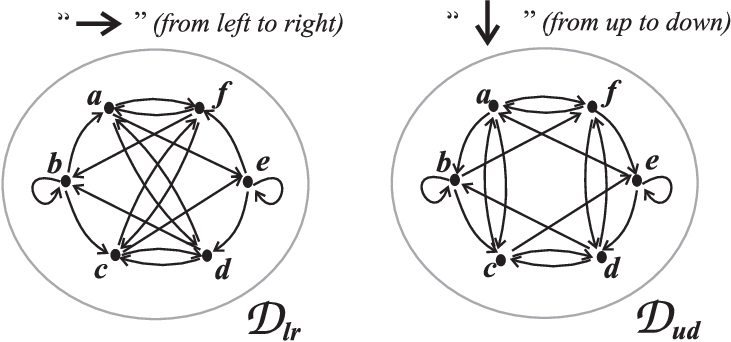}
\\ \ \vspace*{-18pt}
\end{center}
\caption{The  digraphs ${\cal D}_{ud}$ and  ${\cal D}_{lr}$.}
\label{CvorniKod2}
\end{figure}
\unskip

For each alpha-letter $\alpha$, we denote by $\overline{\alpha}$ the alpha-letter
of the situation from  Figure~\ref{CvorniKod1}  obtained
by applying reflection  symmetry with the horizontal  axis as its line of symmetry. Precisely,
  $\ds
\overline{a}  \stackrel{\rm def}{=}  c, \overline{b}  \stackrel{\rm def}{=}  b, \overline{c}  \stackrel{\rm def}{=}  a$,
 $ \overline{d}  \stackrel{\rm def}{=}  f, \overline{e}  \stackrel{\rm def}{=}  e$ and  $\overline{f}  \stackrel{\rm def}{=}  d$.
 Further, for any alpha-word $v = \alpha_{1}\alpha_{2} \ldots \alpha_{m} \in V({\cal D}_{m})$, we introduce   $\overline{v} \stackrel{\rm def}{=}\overline{\alpha}_m \overline{\alpha}_{m-1} \ldots \overline{\alpha}_1 \in V({\cal D}_{m})$.

\begin{thm}   \label{thm:stara} (The characterization of a 2-factor \cite{DjBD1}) \\
The code matrix  $ \ds \left[ \alpha_{i,j}\right]_{m \times n}$ for a given 2-factor of a linear grid graph $G$ of width $m$ ($m \in N$)
 has the following properties:

\begin{enumerate}
\item \textbf{Column conditions:} \\ For every fixed \ $j$ ($1 \leq j \leq n$) the ordered pairs \ $ (\alpha_{i,j}, \alpha_{i+1,j})$
  must be arcs in the digraph \ ${\cal D}_{ud}$  for $1 \leq i \leq m-1$. Additionally,  $ \alpha_{1,j} \in \{ a, d, e \}$ \ and
\ $\alpha_{m,j} \in \{ c, e, f \}$.

\item \textbf{Adjacency of columns condition:} \\  For every  fixed  $j$ ($1 \leq j \leq n-1 $),
 the ordered pairs \ $ (\alpha_{i,j}, \alpha_{i,j+1})$ must be arcs in the digraph \ ${\cal D}_{lr}$  for $1
\leq i \leq m$.

\item \textbf{First and Last Column conditions:}
\begin{enumerate}
 \item
If $G= RG_{m}(n)$, then
the alpha-word of the first  column consists of the letters from the
set \ $\{ a, b, c \}$
%, \ with $\alpha_{1,1}= a$  \ and  \ $\alpha_{m,1}= c$
 and of  the last column of the letters
  from the set \ $\{  b, d, f \}$.

\item
If $G= TkC_{m}(n)$,  then
 the ordered pairs \ $ (\alpha_{i,n}, \alpha_{i,1})$, \ where \ $1
\leq i \leq m$, \ must be arcs in the digraph \ ${\cal D}_{lr}$.

\item
If $G= MS_{m}(n)$, then
 the ordered pairs \ $ (\overline{\alpha}_{i,n}, \alpha_{m -i+1,1})$,  \ where \  $1 \leq i \leq m$,  must be arcs in the digraph \ ${\cal D}_{lr}$.
 \end{enumerate}
\end{enumerate}
The converse, for every matrix $[\alpha_{i,j}]_{m \times n }$ with entries from $\{a,b,c,d,e,f\}$
that satisfies conditions 1--3 there is  a  unique 2-factor  on the considered grid graph $G$.
\end{thm}

Thus the counting of such code matrices for a linear grid graph $G$ of width  $m$ ($m \in N$) is  reduced to  the counting of some  specific  directed walks of fixed length
in an auxiliary  digraph $ {\cal D}_{m}
\stackrel{\rm def}{=} $  $(V({\cal D}_{m}), E({\cal D}_{m}))$.
The  vertices of this digraph are all  possible alpha  words, i.e.  all  possible  words  $\alpha_{1}\alpha_{2} \ldots \alpha_{m}$ over  alphabet $ \{ a,b,c,d,e,f \}$  which fulfill that the ordered pairs  $ (\alpha_{i}, \alpha_{i+1})$
  must be arcs in the digraph \ ${\cal D}_{ud}$  for $1 \leq i \leq m-1$ , $ \alpha_{1} \in \{ a, d, e \}$ \ and
\ $\alpha_{m} \in \{ c, e, f \}$.
An arc $(v,u) $ belongs to  $ E({\cal D}_{m})$  if and only if
the alpha word $v= v_1v_2 \ldots v_m$  can be the previous column of the alpha word $u=u_1u_2 \ldots u_m$ in  a  code  matrix  $ \left[ \alpha_{i,j}\right]_{m \times n}$,
i.e. the ordered pairs \ $ (v_{i}, u_{i})$ must be arcs in the digraph \ ${\cal D}_{lr}$  for $1 \leq i \leq m$ (in accordance to  the Adjacency of columns condition).
 When  $m \geq 2$, the  digraph ${\cal D}_{m}$  is disconnected and  $ \ds   \mid V({\cal D}_{m}) \mid = \ds \frac{3^m + (-1)^m}{2}$ \cite{DjBD1}.

\begin{df} \label{def:outlet} \cite{DjBD1}
The {\bf {\em outlet  word}} of  a vertex $\alpha \equiv \alpha_1 \alpha_2 \ldots \alpha_m \in V({\cal D}_{m})$ is  the  binary word
 $o(\alpha) \equiv o_1o_2 \ldots o_m$, where
  $ \ds o_j
\stackrel{\rm def}{=} \left \{
\begin{array}{cc}{}
0, & \; \; if \; \; \alpha_j \in \{ b, d, f \} \\
1, & \; \; if \; \; \alpha_j \in \{ a, c, e \}
\end{array}
 \right.
 ,  \; \; \; 1 \leq j \leq m.
$ \\
For a  binary word $v \equiv b_1b_2 \ldots b_{m-1}b_{m} \in \{ 0,1\}^m$,
$\ds \overline{v} \stackrel{\rm def}{=}  b_mb_{m-1}$ $ \ldots b_{2}b_{1}$.
\end{df}
\begin{exm}  \label{exm:1}
For the 2-factor    of $RG_7(8)$   depicted in  Figure~\ref{SSL} (a)
the outlet words for  the first three columns  are $1^501$, $1^401^2$  and $1^501$, respectively.
\end{exm}

To obtain the new  digraph \ $ {\cal D}^*_{m} \stackrel{\rm def}{=}(V({\cal D}^*_{m}), E({\cal D}^*_{m}))$ from $ {\cal D}_{m}$  we
glue  all  the vertices  with the same corresponding outlet word and replace all the
arcs starting from these glued vertices and ending
at the   same vertex with only one arc.

The following statements are  proved in  \cite{DjBD1}. Every binary word from $ \{ 0,1 \}^m$ except the word  $(01)^k0$ in case  $m=2k+1$ ($k \in N $) appears
as the outlet word of a vertex in $ {\cal D}_{m}$ and therefore belongs to  $V({\cal D}^*_{m})$.
 For $m \geq 2$  digraphs ${\cal D}^*_{m}$ are disconnected (a vertex, considered as an outlet word, with odd number of 1's  can not be in the same component with the one which has even number of 1's). Each  component of ${\cal D}^*_{m}$ is a strongly connected digraph, i.e.
its  adjacency matrix  ${\cal T}^*_{m}$  is a symmetric binary matrix.

\begin{thm}   \label{thm:stara1}(\cite{DjBD1}) \
If  $f_m^{RG}(n)$,  $f_{m}^{TkC}(n)$ and  $f_{m}^{MS}(n)$ ($ m \geq 2$) denote the number of 2-factors of $RG_{m}(n)$, $TkC_{m}(n)$ and  $MS_{m}(n)$, respectively, then \
$$ \ds
f_m^{RG}(n)=  a_{1,1}^{(n)}, $$
\bc
 $ \ds
f_{m}^{TkC}(n) = \ds \sum_{\begin{array}{c} v_i \in   V({\cal D}_{L,m}^*)\end{array} }  a_{i,i}^{(n)} $
 \ \  and \ \ $ \ds
f_{m}^{MS}(n) =
   \sum_{\begin{array}{c} v_i, v_j  \in   V({\cal D}_{L,m}^*) \\
\overline{v_i}= v_{j} \end{array} } a_{i,j}^{(n)} , $
\ec
 where      $v_1 \equiv 0^m$ (corresponding to the first row and first column of ${\cal T}^*_{m}$) and
 $ a_{i,j}^{(n)}$ denotes the $(i,j)$-entry of $n$-th power of ${\cal T}^*_{m}$.
 \end{thm} \noindent

The computational  data for the digraphs ${\cal D}^*_m$ where $m \leq 12$ were  collected  in the same paper
and  pointed to the structure of ${\cal D}^*_{m}$ which is expressed in the following
theorem and  proved in \cite{DjDB2}.

\begin{thm}  \label{conj:1}  \cite{DjDB2}
For each  $m \geq 2 $,  the digraph  ${\cal D}^{*}_{m}$ has exactly $\ds \left\lfloor    \ds \frac{m}{2} \right\rfloor + 1$ components, i.e.
 $\ds {\cal D}^{*}_{m} = {\cal A}^*_{m}    \cup $ $\ds  (\bigcup_{s=1}^{\left\lfloor    \frac{m}{2} \right\rfloor }{\cal B}^{*(s)}_{m})$, where
 $\ds \mid V({\cal B}^{*(1)}_{m}) \mid  \geq
  \mid V({\cal B}^{*(2)}_{m}) \mid  \geq  \; \; \;  \ldots \; \; \;  \geq \mid V( {\cal B}^{*(\lfloor    m/2 \rfloor )}_{m}) \mid $
 and   ${\cal A}^*_{m} $ is the one containing $1^m$.
All the components  ${\cal B}^{*(s)}_{m}$ ($ 1 \leq s \leq \ds \left\lfloor    \ds \frac{m}{2} \right\rfloor $) are bipartite digraphs.
\\
 If  $m$ is  odd, then $\ds \mid V({\cal B}^{*(s)}_{m}) \mid  =\ds  {m + 1 \choose   (m+1)/2 -s} \mbox{ \  and \ } \ds \mid V({\cal A}^{*}_{m}) \mid =  \ds {m  \choose (m-1)/2}.$
\\
  If  $m$ is even, then $\ds \mid V({\cal B}^{*(s)}_{m}) \mid  = \ds 2 {m \choose   m/2 -s} \mbox{ \  and \  } \ds \mid V({\cal A}^{*}_{m}) \mid =  \ds {m \choose m/2}.$ \\
    The vertices $v$ and $\overline{v}$ belong to  the same component. When the component is bipartite they are placed in the same  class  if and only if  $m$ is odd.
\end{thm}

The component of ${\cal D}_{m}$ which contains all the possible columns of code matrices for $RG_m(n)$   ($m \geq 2$) is denoted by ${\cal R}_m$.
The component of ${\cal D}^{*}_{m}$ asigned to ${\cal R}_m$ is denoted  by ${\cal R}^{*}_m$. It contains the vertex $0^m$ and it is sufficient for counting  2-factors of $RG_m(n)$.
Using the  property of reflection symmetry we have that $ \overline{v} \in {\cal R}^*_{m}$ for every vertex  $v  \in {\cal R}^*_{m}$.
Since $\overline{0 0 \ldots 0} = 0 0 \ldots 0 \in {\cal R}^*_{m}$ further reduction of the graph ${\cal R}^{*}_m$ is possible.
By  gluing the vertices $v$ and $\overline{v}$   into one vertex for all $v \in V({\cal R}^*_m)$ the set of vertices of new  digraph ${\cal R}_m^{**}$ is obtained.
During this process the  arcs starting from just one of these two vertices are retained.
Multiple (double) arcs appear  when $v$ and $\overline{v}$ have a  common direct predecessor.
In this way, we obtain a lower-order transfer matrix.

\begin{thm}  \label{thm:contraction} \cite{DjBD1} \\ The number
$f_m^{RG}(n)$ is equal to entry $a_{1,1}^{(n)}$ of the $n$-th power of the adjacency matrix for  ${\cal R}^{**}_m$ where   $v_1 \equiv 0^m$.
\end{thm}

For even $m$, it is proved in  \cite{DjBD1} (Theorem 3.18) that  all palindromes from $V({\cal D}^{*}_m)$ belong to the component ${\cal A}^{*}_m$ which coincides with ${\cal R}_m^*$.
Consequently, using    Theorem~\ref{conj:1}
we have
 \be \label{e1} \mid V({\cal R}^{*}_{m}) \mid = {m \choose m/2 }  \ee  while   the number of vertices in    digraph ${\cal R}^{**}_{m}$ is equal  to  \mbox{  A005317 in OEIS}, i.e.
 \be \label{e2} \mid  V({\cal R}^{**}_{m}) \mid = 2^{(m-2)/2}+ \frac{1}{2}{m \choose m/2}. \ee
The equality (\ref{e1}) and (\ref{e2}) are a part of Conjecture 4.3 in  \cite{DjBD1}.

The  aim of  this paper is the first proof of  the remaining part of this conjecture and of Conjecture 4.2 in  \cite{DjBD1} which refer to the case when $m$ is odd.
We will prove
\begin{thm}  \label{cor:2} (MAIN THEOREM)
 For $m-odd$, the number of vertices in    ${\cal R}^{*}_{m}$ is equal  to the  binomial coefficient  (in OEIS A001791):
 \begin{eqnarray} \label{e0} \mid V({\cal R}^{*}_{m}) \mid =  {m+1 \choose (m-1)/2 }  \end{eqnarray}
    while   the number of vertices in    digraph ${\cal R}^{**}_{m}$ is equal  to
    \begin{eqnarray} \label{e4} \mid  V({\cal R}^{**}_{m}) \mid =  \frac{1}{2} \left[{m+1 \choose (m-1)/2 } + {(m+1)/2  \choose \lfloor (m+1)/4  \rfloor}\right]. \end{eqnarray}
   \end{thm}

\newpage
\begin{center}
{\bf 2. PROOF OF  THE MAIN THEOREM}
\end{center}

We need two definitions.

\begin{df} \label{def:ch} \cite{DjDB2}
For a binary word $x $ of length $m$ ($ m \in N$) we denote by $odd(x)$ ($even(x)$) the total number of 0's at odd (even) positions in $x$.
The difference  $odd(x) - even(x)$ is labeled as $Z(x)$.
\end{df}

\begin{df} \label{def:ch} \cite{DjDB2}
The set  $S_{m}^{(0)}$ ($m \in N$) consists of all the  binary $m$-words whose number of 0's at odd positions is equal to the  number of 0's at even  positions.
For  $1 \leq s \leq \lfloor m/2 \rfloor$,  $S_m^{(s)}\stackrel{\rm def}{=}  R_m^{(s)} \cup  G_m^{(s)}$  where
the words in   $R_{m}^{(s)}$  and $G_{m}^{(s)}$ are all the  binary words $x$ of the length $m$
for which  $ Z(x)  = s$ and  $ Z(x) = -s$, respectively. Additionally, if  $m$ is odd, then $R_m^{(\lceil m/2 \rceil)}\stackrel{\rm def}{=}\{ 0 (10)^{\lfloor m/2 \rfloor} \}$.
\end{df}

\begin{exm}  \label{exm:3}
It is easy to check that $ S_1^{(0)} = \{1 \}$, $R_1^{(1)} = \{0 \}$,
$S_2^{(0)}  = \{ 00, 11 \}$, $R_2^{(1)} =\{ 01 \}$,  $ G_2^{(1)}  =\{ 10 \} $,
$ S_3^{(0)} =\{ 100, 111, 001  \},   R_3^{(1)}=\{ 000, 011, 110 \},  G_3^{(1)}= \{ 101 \}, R_3^{(2)}  = \{010 \}$.
 \end{exm}

Note that $\ds \bigcup_{s=0}^{\lfloor m/2 \rfloor} S_m^{(s)}  =  V({\cal D}_{m}^*)$, where $ V({\cal D}_{m}^*) = \{ 0,1\}^m$ for $m$-even, and
 $ V({\cal D}_{m}^*) = \{ 0,1\}^m \backslash R_m^{(\lceil m/2 \rceil)} = \{ 0,1\}^m \backslash \{ 0 (10)^{\lfloor m/2 \rfloor} \}$ for $m$-odd.
In \cite{DjDB2}, it is proved that the subdigraphs of ${\cal D}^{*}_{m}$ induced  by the sets  $S_{m}^{(s)}$ ($0 \leq s \leq \lfloor m/2 \rfloor$)
are its  components. To be more accurate, $ \langle S_{m}^{(0)} \rangle_{{\cal D}^{*}_{m}} = {\cal A}^*_{m}$  and  $ {\cal B}^{*(s)}_{m}  = \langle S_{m}^{(s)} \rangle_{{\cal D}^{*}_{m}} $,  where
$ R_m^{(s)} $ and $G_m^{(s)}$ are the  classes of the bipartite digraph $ {\cal B}^{*(s)}_{m} $  ($1 \leq s \leq \lfloor m/2 \rfloor$).  The vertices  from $ R_m^{(s)} $ and $G_m^{(s)}$ ($1 \leq s \leq \lfloor m/2 \rfloor$) are called \textbf{ red } and \textbf{ green } ones, respectively.

\begin{lem}   \label{lem0} \cite{DjDB2} If $1 \leq s \leq k $ ($k \in N$), then
\\ $  \ds \mbox{ a)} \ds \mid R^{(s)}_{2k} \mid  = \mid G^{(s)}_{2k} \mid = \ds    {2k \choose  k-s}, \mbox{ \ }
\ds \mid S^{(s)}_{2k} \mid  = 2 \cdot {2k \choose  k-s} $,
$\ds \mid S^{(0)}_{2k} \mid = \ds  {2k \choose   k}. $
\\
\\ $ \mbox{ b)} \ds \mid R^{(s)}_{2k+1} \mid  =\ds  {2k +1 \choose   k-s+1}$,   $\ds \mid G^{(s)}_{2k+1} \mid = \ds  {2k +1 \choose   k-s}, $
$\ds \mid S^{(s)}_{2k+1} \mid = \ds  {2k +2 \choose   k-s+1} $, $ \ds \mid S^{(0)}_{2k+1} \mid  = \ds  {2k+1 \choose   k}.$
\end{lem}

The following two lemmas are  proposed in \cite{DjBD1} as Conjecture 4.2.

\begin{lem}  \label{lem1}
 For  odd $m$,  ${\cal R}^{*}_{m} \equiv  {\cal B}^{*(1)}_{m}$.
 \end{lem}
\begin{proof} Since  $m$ is odd, we have  $Z(0^m) = 1$, which implies  $0^m \in  V( {\cal B}^{*(1)}_{m})$.  Consequently,  ${\cal R}^{*}_{m}\equiv  {\cal B}^{*(1)}_{m}$. $\Box$ \
\end{proof}

\begin{lem}  \label{lem2}
 For  odd $m$, the number of all palindromes from $ V( {\cal B}^{*(1)}_{m})$ is equal to $\ds {(m+1)/2  \choose \lfloor (m+1)/4  \rfloor}$.
 \end{lem}
 \begin{proof}
Consider a palindrome from $ V( {\cal B}^{*(1)}_{m})$ where  $m=2k+1$, $ k \in N$. Note that its forme must be $w0\overline{w}$. \\
\underline{{\bf Case I:}} \  If $k$ is even, then  $Z(\overline{w}) = -Z(w)$ and $Z(w0\overline{w}) = Z(w) + 1 -Z(\overline{w}) = 2 Z(w) + 1$. Since $\mid Z(w0\overline{w})\mid = 1$, we have that
$Z(w)= -1$ or $Z(w)= 0$. It implies that $  w \in G_{k}^{(1)} \cup   S_{k}^{(0)}$. \\
From  Theorem~\ref{conj:1} and Lemma~\ref{lem0}, since  $k$  is even, we obtain that
the number of all palindromes is

$\ds  \mid G_{k}^{(1)} \cup   S_{k}^{(0)} \mid = {k \choose k/2 -1} + {k \choose k/2} = {k+1 \choose  k/2} = {(m+1)/2 \choose  \lfloor(m+1)/4 \rfloor}$.
\\
\\
\underline{{\bf Case II:}} \  If $k$ is odd, then  $Z(\overline{w}) = Z(w)$ and $Z(w0\overline{w}) = Z(w) - 1 +Z(\overline{w}) = 2 Z(w) - 1$. Since $\mid Z(w0\overline{w})\mid = 1$, we have that
$Z(w)= 0$ or $Z(w)= 1$. It implies that $  w \in   S_{k}^{(0)}  \cup  R_{k}^{(1)}$. \\
Applying Theorem~\ref{conj:1} and Lemma~\ref{lem0} again,
 we obtain that the number of all palindromes is

 $\ds  \mid S_{k}^{(0)} \cup   R_{k}^{(1)} \mid = {k \choose \lfloor k/2 \rfloor } + {k \choose \lfloor k/2 \rfloor} = {k+1 \choose   \lfloor (k +1)/2 \rfloor } = {(m+1)/2 \choose  \lfloor(m+1)/4 \rfloor }$, in both cases the same.
 $\Box$ \end{proof}

Now, equation \eqref{e0} is a  trivial consequence of Lemma~\ref{lem1} and Theorem~\ref{conj:1}.  Namely,
for $m$-odd,
$$ \mid V({\cal R}^{*}_{m}) \mid =  \mid V({\cal B}^{*(1)}_{m}) \mid  = {m+1 \choose (m+1)/2 -1} = {m+1 \choose (m-1)/2 }. $$

In order to obtain
 the number of vertices in    digraph ${\cal R}^{**}_{m}$
 note that this number is equal to  the number of vertices in ${\cal R}^{*}_{m}$ decreaseded by the half of the number of its vertices which are not palindromes.
 Thus, using just proved  (\ref{e0}) and  Lemma~\ref{lem2} we obtain
  $$ \mid V({\cal R}^{**}_{m}) \mid  = \mid  V({\cal R}^{*}_{m}) \mid - \frac{1}{2}\left[ \mid  V({\cal R}^{*}_{m}) \mid -  {(m+1)/2  \choose \lfloor (m+1)/4  \rfloor} \right] =
  \frac{1}{2}  \left[{m+1 \choose (m-1)/2 } + {(m+1)/2  \choose \lfloor (m+1)/4  \rfloor}\right]. $$
 In this way we have  completed  the   proof of main theorem. $\Box$

\vspace*{1cm}

\noindent
REMARK. \
We can offer the shorter proof for the
Theorem 3.18 in  \cite{DjBD1} which claims that all palindromes from $ V({\cal D}^{*}_{m}) $ when $m$ is even belong to the component   $ {\cal A}^{*}_{m}$. Namely,
if $w$ is a palindrome ($w = \overline{w} $), then $Z(w) =  Z(\overline{w})$. Since $m$ is even, we have that $Z(\overline{w}) = -Z(w)$. Consequently, $Z(w) =-Z(w)$, i.e. $ Z(w)=0$.
It implies that $w \in \langle  S_{m}^{(0)} \rangle_{{\cal D}^{*}_{m}} =  {\cal A}^{*}_{m}$. $\Box$

% ===============
% Acknowledgement
% ===============

\vspace{1.5cc}
\begin{center}
{\bf ACKNOWLEDGEMENTS}
\end{center}

%The authors are indebted to the anonymous referees  for their valuable
%suggestions and helpful comments which  improved the clarity of the presentation.

This work  was  supported by   the Project of the Department for fundamental disciplines in technology, Faculty of Technical Sciences, University of Novi Sad "Application of general disciplines in technical and IT sciences".

% ==========
% References
% ==========

%\vspace*{0.5cm}
%Fill author(s) affiliation(s), address(es) and emails here:

\noindent Faculty of Technical Sciences,
  University of Novi Sad,
  Novi Sad, Serbia\\
     E-mail: jelenadjokic@uns.ac.rs
\vspace*{0.5cm}

\end{document}